\documentclass{amsart}

\usepackage[pagebackref=true,colorlinks=true, pdfstartview=FitV, linkcolor=blue,citecolor=red, urlcolor=blue]{hyperref}

\usepackage{amsmath,amsthm,amssymb,mathrsfs,amscd}
\usepackage{color,enumerate}
\usepackage{hyperref}
\usepackage{url}
\usepackage[all]{xy}
\usepackage[top=30truemm,bottom=30truemm,left=25truemm,right=25truemm]{geometry}

\newcommand{\bburl}[1]{\textcolor{blue}{\url{#1}}}

\newtheorem{thm}{Theorem}[section]

\newtheorem{lem}[thm]{Lemma}
\newtheorem{prop}[thm]{Proposition}
\theoremstyle{definition}

\theoremstyle{definition}

\theoremstyle{remark}
\newtheorem{rem}[thm]{Remark}

\numberwithin{equation}{section}

\newcommand{\B}{\ensuremath{{\bf B}}}

\newcommand{\GL}{\mathrm{GL}}

\newcommand{\Hom}{\mathrm{Hom}}

\title{The Siegel-Weil formula for unitary groups: the second term range}

\author{Hengfei Lu}
\email{\textcolor{blue}{\href{hengfei.lu@weizmann.ac.il}{hengfei.lu@weizmann.ac.il}}}
\address{Department of Mathematics, Weizmann Insitute of Science, 234 Herzl St. POB 26, Rehovot 7610001, Israel}

\date{\today}

\begin{document}

\begin{abstract}
We study the Siegel-Weil formula in the second term range ($n+1\leq m\leq n+r$) for unitary groups of hermitian forms over a skew-field $D$ with involution of the second kind.
\end{abstract}

\maketitle

\tableofcontents
\section{Introduction}
The Siegel-Weil formual is an identity between an Eisenstein series and an integral of a regularized theta function. The convergent case ($r=0$ or $m>n+r$) was studied first by Weil in \cite{weil}. The case for the classical unitary group ($d=1$) have been extensively studied by Ichino \cite{ichino2001crelle,ichino2004mathz,ichino2007sw} and Gan-Qiu-Takeda \cite{gan2014siegelweil}. When $d>1$, the first term identity in the first term range ($m\leq n$) was proved by Yamana in \cite{yamana2013siegel}. This paper will focus on the sencond term range, i.e., $n+1\leq m\leq n+r$ and $d>1$. The proof is indebted to Gan-Qiu-Takeda \cite{gan2014siegelweil}.

Following \cite{yamana2013siegel}, let $E/F$ be a quadratic extension of number fields and $D$ be a division algebra with center $E$, of dimension $d^2$ over $E$ and provided with an antiautomorphism $\ast$ of order two under which $F$ is the fixed subfield of $E$. Let $\mathbb{A}$ and $\mathbb{A}_E$ be the adele rings of $F$ and $E$ respectively. Let $\omega_{E/F}$ be the quadratic charater of $\mathbb{A}^\times/F^\times$ associated to the extension $E/F$.
 Given a local place $v$ of $F$, let $F_v$ be the $v$-completion of $F$ and set $E_v=E\otimes_F F_v,D_v=D\otimes_F F_v$. Then
\[D_v\cong\begin{cases}
M_{d}(E_v)&\mbox{ if }E_v\mbox{ is a local field},\\
D_{F_v}\oplus D_{F_v}^{op}&\mbox{ if }E_v=F_v\oplus F_v,
\end{cases} \]
where $D_F$ is a central simple algebra with center $F$, of dimension $d^2$ over $F$, $D_{F_v}=D_F\otimes_F F_v$ is a central simple algebra with center  $F_v$ and $D_{F_v}^{op}$ is its opposite algebra (see \cite[Theorem 10.2.4]{scharlau1985}).
 Let $W_{2n}$ be a right $D$-vector space of dimension $2n$ with a nondegenerate skew-Hermitian form that has a complete polarization, and $V_r$ a left $D$-vector space of dimension $m$ with a nondegenerate Hermitian form. Let $\chi_V$ be the quadratic character of $\mathbb{A}_E^\times/E^\times$ associated to $V$ such that $\chi_V|_{\mathbb{A}^\times/F^\times}=\omega_{E/F}^{dm}$.
  Let $V_0$ be a left $D$-vector space of dimension $m_0$ with $U(V_0)$ anisotropic and $V_r=V_0\oplus D^{2r}$, where $D^{2r}$ is a $D$-vector space with Hermtian form
\[\langle x,y\rangle =x J(y^\ast)^t,J=\begin{pmatrix}
0&\mathbf{1}_n\\\mathbf{1}_n&0
\end{pmatrix} \] for $x,y\in D^{2r}$, $\mathbf{1}_n$ is the identity matrix in $M_n(D)$, $x^t$ is the transpose of $x$ and $r$ is called the Witt index of $V_r$.  Let $G_{2n}$ and $H_r$ be the unitary group of $W$ and $V$ respectively. Then
\[G_{2n}(F_v)\cong\begin{cases}
U_{nd,nd}&\mbox{ if }E_v\mbox{ is  a field};\\
\GL_{2n}(D_{F_v})&\mbox{ if }E_v=F_v+F_v.
\end{cases} \]

Let $\alpha_E$ denote the standard norm of 􏳗$\mathbb{A}^\times_E$ .  We denote by $P$ the maximal parabolic subgroup
of $G_{2n}$ that stabilizes a maximal isotropic subspace of $W$􏳞. Note that $P$ has a Levi
decomposition $P = MN$ with $M \cong \GL_n(D)$.  For any unitary character $\chi$ of 􏳗$\mathbb{A}^\times_E /E^\times$
and for any $s\in\mathbb{C}$ 􏳘,we consider the representation $I(s,\chi)=Ind_{P(\mathbb{A})}^{G_{2n}(\mathbb{A})}\chi\alpha_E^s$
 induced
from the character $m 􏰀\rightarrow \chi(\nu(m))\alpha_E(\nu(m))^s$, where $\nu$ is the reduced norm viewed
as a character of the algebraic group $\GL_n(D)$ and the induction is normalized so that $I (s, \chi)$ is naturally unitarizable when $s$ is pure imaginary.
When $E=F+F$, we consider $$I(s,\mathbf{1})=Ind_{P(\mathbb{A})}^{GL_{2n}(D_F(\mathbb{A}))}\alpha_E^s\boxtimes\alpha_E^{-s},$$ where $P=MN$ and $M\cong \GL_n(D_F)\times\GL_n(D_F)$.
 For any holomorphic section $f^{(s)}$ of $I(s,\chi)$, i.e.
\[f^{(s)}(mng)=\chi(\nu(m))\alpha_E(\nu(m))^{s+dn/2}f^{(s)}(g) \]
for $m\in\GL_n(D(\mathbb{A}))$, $n\in N(\mathbb{A})$ and $g\in G_{2n}(\mathbb{A})$,
 the Siegel Eisenstein series
\[E(g;f^{(s)})=\sum_{\gamma\in P(F)\backslash G(F) }f^{(s)}(\gamma g)  \]
is absolutely convergent for $Re(s)>\frac{dn}{2}$
and has a meromorphic continuation to the whole $s$-plane.
\begin{lem}\cite[Theorem 1]{yamana2013siegel}\label{siegeleisenstein}
	If $n+1\leq m\leq n+r$ and $r>0$, then the Siegel Eisenstein series $E(g;f^{(s)})$ has a simple pole at $s=s_0=(m-n)d/2$ where $\chi=\chi_V$.
\end{lem}

Fix a nontrivial additive character $\psi$ of $\mathbb{A}/F$ and a character $\chi_V$ of $\mathbb{A}_E^\times/E^\times$ such that $\chi_V|_{\mathbb{A}^\times}=\omega_{E/F}^{dm}$.
The group $G_{2n}(\mathbb{A}􏳗)\times H_r(\mathbb{A}􏳗)$ acts on the Schwartz space $\mathfrak{S}(V_r^n(\mathbb{A}􏳗)) $
of $V_r^n(\mathbb{A}􏳗)$ via the Weil representation $\omega_{n,r}$. Let $S(V_r^n(\mathbb{A}􏳗))$ be the subspace of $\mathfrak{S}(V_r^n(\mathbb{A}􏳗))$ consisting of functions that correspond to polynomials in the Fock model at every archimedean place of $F$. Given a function  $\phi\in S(V_r^n(\mathbb{A}))$,
set $$\Phi^{n,r}(\phi)(g)=\omega(g)\phi(0).$$ Then $\Phi^{n,r}(\phi)\in I((m-n)d/2,\chi_V)$ which is called the Siegel-Weil section associated to $V_r$. Suppose $f^{(s)}=\Phi^{n,r}(\phi)$ and the Siegel Eisentein series has an expression
\[E(s,\Phi^{n,r}(\phi))= \sum_{j\geq-1}A^{n,r}_j(\phi)(s-s_0)^j, \]
where each Laurent coefficient $A^{n,r}_j(\phi)$ is an automorphic form on $G_{2n}$ and $A_j^{n,r}$ can be viewed as a linear map
\[A_j^{n,r}:\omega_{n,r}\longrightarrow \mathcal{A}(G_{2n}) \]
where $\mathcal{A}(G_{2n})$ is the space of automorphic forms on $G_{2n}$.

The theta function associated to 􏳔 $\phi\in S(V_r^n(􏳗\mathbb{A}))$ is defined by
\[\Theta(g,h;\phi)=\sum_{x\in V_r^n(F)}(\omega(g)\phi)(h^{-1}x)  \]
for $g\in G_{2n}(\mathbb{A})$ and $h\in H_r(\mathbb{A})$. Let $\tau(H_r)$ denote the Tamagawa number of $H_r$. By Weil's criterion \cite{weil}, the integral
\[I_{n,r}(\phi)(g)=\frac{1}{\tau(H_r)}\int_{H_r(F)\backslash H_r(\mathbb{A}) }\Theta(g,h;\phi)dh  \]
is absolutely convergent for all $\phi$ either if $r=0$ or $m>r+n$.
When $m\leq r+n$ and $r>0$, the integral diverges in general.

Let $V_r=X_r\oplus V_0\oplus X_r^\ast$ such that $X_r$ is the maximal isotropic subspace  in $V$ and $U(V_0)$ is anisotropic. Let $P(X_r)=M(X_r)N(X_r)$ be the maximal parabolic subgroup of $H_r$ which stabilizes the spaces $X_r$. Then its Levi factor is 
\[M(X_r)\cong\GL_r(D)\times U(V_0). \] 
Let us fix the Iwasawa decomposition $$H_r(\mathbb{A})=P(X_r)(\mathbb{A})\cdot K_{H_r}$$ such that $K_{H_r}\cap \GL(X_r)(\mathbb{A})$ is a maximal compact subgroup of $\GL_r(X_r)(\mathbb{A})$.
Let $$I_{H_r}(s)=Ind_{P(X_r)(\mathbb{A})}^{H_r(\mathbb{A})}\alpha_E^s\boxtimes\mathbf{1}_{U(V_0)}$$ be the normalized induced representation of $H_r(\mathbb{A})$ where $\alpha_E^s$ is a character of $\GL_r(D(\mathbb{A}))$ and $\mathbf{1}_{U(V_0)}$ is the trivial representation of $U(V_0)$.

 Following \cite{kudla1994annals}, Ichino \cite{ichino2001crelle} defined a regularization of the integral $I(g,\phi)$ as follow
\begin{equation}\label{regularization}
\mathcal{E}^{n,r}(s,\phi)(g)=\frac{1}{\tau(H_r)\cdot \kappa_r\cdot P_{n,r}(s)}\int_{H_r(F)\backslash H_r(\mathbb{A}) }\Theta(g,h;z.\phi)E_{H_r}(s,\phi)dh, \end{equation}
where 
\begin{itemize}
	\item $z$ lies in the spherical Hecke algebra of $G_{2n}(F_v)\cong U_{nd,nd}$ for $v$ non-archimedean and $E_v$ is a field so that the action of $z$ commutes with the action of $G_{2n}(\mathbb{A})\times H_r(\mathbb{A})$ and $\Theta(g,-;z.\phi)$ is rapidly decreasing;
	\item $E_{H_r}(s,h)$ is the Eisenstein series given by 
	\[E_{H_r}(s,h)=\sum_{\gamma\in P(X_r)(F)\backslash H(F) }f_s^0(\gamma h)  \]
	where $f_s^0\in I_{H_r}(s)$ is the $K_{H_r}$-spherical standard section with $f_s^0(1)=1$;
	\item $P_{n,r}(s)$ is a scalar such that the Hecke operator $z\ast E_{H_r}(s,-)=P_{n,r}(s)\cdot E_{H_r}(s,-)$, which can be found in \cite[Page 208]{ichino2001crelle}.
\end{itemize}
The regularized integral (\ref{regularization}) converges absolutely at all points $s$ where $E_{H_r}(s,h)$ is holomorphic, and defines a meromorphic function of $s$ (independent of the choice of the Hecke operator $z$). (See \cite{ichino2001crelle}.) We are interested in the behavior of $\mathcal{E}^{n,r}(s,\phi)$ at
\[s=\rho_{H_r}=(m-r)d/2. \]
It turns out that in the first term range, when $m\leq n$, it has a pole of order at most $1$ 
  whereas in the second term range, it has a pole of order at most $2$ when $n+1\leq m\leq n+r$ and $r>0$.
Thus, the Laurent expansion of (\ref{regularization}) at $s=\rho_{H_r}$ has the form
\[\mathcal{E}^{n,r}(s,\phi)=\sum_{i\geq-2}B^{n,r}_i(\phi)(s-\rho_{H_r})^i  \]
where $B^{n,r}_{-2}(\phi)=0$ if $m\leq n$.
Then each Laurent coefficient $B^{n,r}_i(\phi)$ is an automophic form on $G_{2n}$, and hence we view $B^{n,r}_i$ as a linear map
\[B^{n,r}_i:\omega_{n,r}\longrightarrow \mathcal{A}(G), \]
via $\phi\mapsto B^{n,r}_i(\phi)$,
where $\mathcal{A}(G_{2n})$ is the space of automorphic forms on $G_{2n}$.

 Yamana \cite{yamana2013siegel} showed the first term identity in the first term range, i.e. $m\leq n$. In this paper, we will focus on the sencond term range, i.e. $n+1\leq m\leq n+r$.
 \begin{thm}[Siegel-Weil formula] Suppose that $n+1\leq m\leq n+r$. Then one has:
 	\begin{enumerate}[(i)]\label{secondtermsiegelweil}
 		\item (First term identity) $A^{n,r}_{-1}(\phi)=c\cdot B^{n,r}_{-2}(\phi)$ for a constant $c>0$;
 		\item (Second term identity) $$A^{n,r}_0(\phi)=B^{n,r}_{-1}(\phi)+c'\cdot B_0^{n,r'}(Ik^{n,r}(\pi_{K_{H_r}}\phi))\pmod{\mbox{Im } A_{-1}^{n,r}}.$$
 		Here $c'$ is a constant and $0<r'<r$ is such that $m_0+2r'=2n-m$. Moreover,
 		\[Ik^{n,r}:\omega_{n,r}\longrightarrow\omega_{n,r'} \]
 		is the Ikeda map which is $G_{2n}\times H_{r'}$-equivariant. If $m= n+r$, then
 		\[A_0^{n,r}(\phi)=B_{-1}^{n,r}(\phi)\pmod{\mbox{Im }A_{-1}^{n,r} }. \]
 	\end{enumerate}
 \end{thm}

\begin{rem}
	When $d=1$,  it has been proven by Gan-Qiu-Takeda in \cite[Theorem 1.1]{gan2014siegelweil} and $c=1$.
\end{rem}

Now we briefly describe the contents and the organization of this paper. The basic notation will be set up in \S2. In \S3, we will introduce the Eisenstein series and their various properties.
The proof of Theorem \ref{secondtermsiegelweil} will be given in \S4. We will use the doubling method to sudy the nonvanishing of the global theta lift in the last section.

\section{Preliminaries}
From now on, we will follow the notation of Gan-Qiu-Takeda \cite{gan2014siegelweil} in this section. Let $W_{2n}$ be a $2n$-dimensional right $D$-vector space with a nondegenerate skew-Hermitian form. Assume that $Y_n$ is a maximal isotropic subspace in $W_{2n}$ of dimension $n$, so that
$W_{2n}=Y_n\oplus Y_n^\ast$. We fix an ordered basis $\{y_1,y_2,\cdots,y_n \}$ of $Y_n$ and corresponding dual basis $\{y_1^\ast,\cdots,y_n^\ast \}$ of $Y_n^\ast$, so that $Y_n=\oplus^n_{i=1} y_iD$
and $Y_n^\ast=\oplus_{i=1}^ny_i^\ast D$. For any subspace $Y_r=\oplus_{i=1}^ry_iD\subset Y_n$, let
\[Q(Y_r)=L(Y_r)\cdot U(Y_r) \]
denote the maximal parabolic subgroup fixing $Y_r$. Then its Levi factor is
\[L(Y_r)\cong \GL(Y_r)\times G_{2n-2r}. \]
If $r=n$, then $Q(Y_r)$ is a Siegel parabolic subgroup of $G_{2n}$.

The unipotent radical $U(Y_r)$ of $Q(Y_r)$ sits in a short exact sequence
\[\xymatrix{1\ar[r]&Z(Y_r)\ar[r]&N(Y_r)\ar[r]&Y_r\otimes V_{n-r}\ar[r]&1 } \]
where
\[Z(Y_r)=\{\mbox{Hermitian forms on }Y_r^\ast \}\subset \Hom(Y_r^\ast,Y_r). \]

\subsection{Measures} Let us fix the additive character $\psi$ of $\mathbb{A}/F$ and the Tamagawa measure $dx$ on $\mathbb{A}$. Locally, we fix the Haar measure $dx_v$ on $F_v$ to be self-dual with respect to $\psi_v$. For any algebraic group $G$ over $F$, we always use the Tamagawa measure on $G(\mathbb{A})$ when $G(\mathbb{A})$ is unimodular. This applies to the Levi subgroups and the unipotent radical of their parabolic subgroups. We use $\tau(G)$ to denote the Tamagawa number of $G$. 
 For any compact group $K$, we always use the Haar measure $dk$ with respect to which $K$ has volume $1$.
\subsection{Complementary spaces}
With $W_{2n}$ fixed, one may associate to $V_r$ a complementary space $V_{r'}$ such that
\[\dim V_{r'}=m_0+2r'=2n-m \]
and the quadratic character associated to $V_{r'}$ is $\chi_V$. If $m=n+r$, then $r'=0$ and the unitary group $U(V_0)$ is anisotropic. If $r>r'$,
we may write $V_r=X_{r-r'}'\oplus V_{r'}\oplus {X'}_{r-r'}^\ast$ where
\[X_{r-r'}'=\oplus_{i=r'+1}^r  Dx_{i} \]
when $X_r=\oplus_{i=1}^r Dx_{i} $ and $\{x_1,\cdots,x_r \}$ is a basis of $X_r$. We say that $V_r$ and $V_{r'}$ lie in the same Witt tower.

For any maximal parabolic subgroup $P(X_{r-r'})$ of $H_r$, with Levi subgroup $\GL(X_{r-r'})\times H_r'$, we define a constant $\kappa_{r,r'}$ by the requirement that 
\[\frac{1}{\tau(H_r)}=\kappa_{r,r'}\cdot\frac{1}{\tau(H_{r'})}\cdot dm\cdot dn\cdot dk \]
where $dm$ and $dn$ are the Tamagawa measures of $M(X_{r-r'})$ and $N(X_{r-r'})$ respectively. In particular $\kappa_r=\kappa_{r,0}$.
\subsection{Ideka's map}
Suppose that $V_r\supset V_{r'}$(not necessarily complementary spaces) and
\[ \dim V_r=m_0+2r=\dim V_{r'}+2(r-r'). \] 
Then one may write
\[ V_r=X_{r-r'}'\oplus V_{r'}\oplus {X'}^\ast_{r-r'}. \]
We can define a map
\[Ik^{n,r,r'}:S(Y_n^\ast\otimes V_r )(\mathbb{A}) \longrightarrow S(Y_n^\ast\otimes V_{r'})(\mathbb{A}) \]
given by
\[Ik^{n,r,r'}(\phi)(a)=\int_{(Y_n^\ast\otimes X'_{r-r'} )(\mathbb{A}) }\phi(x,a,0)dx,  \]
for $a\in(Y_n^\ast\otimes V_{r'})(\mathbb{A})$. Thus,
$Ik^{n, r,r'}$ is the composite
\[ 
 \begin{CD}
S(Y_n^* \otimes V_r)
= S(Y_n^* \otimes V_{r'})\otimes 
S(Y_n^* \otimes (X'_{r-r'} + {X'}^*_{r-r'}))\\
@VVId \otimes \mathcal{F}_1V \\
S(Y_n^* \otimes V_{r'})\otimes
S(W_{2n}\otimes X'_{r-r'})\\
@VVId \otimes ev_0V \\
S(Y_n^* \otimes V_{r'})
\end{CD}
\]
where 
\[  \mathcal{F}_1 : S(Y_n^* \otimes (X'_{r-r'} +
{X'}^*_{r-r'}))
\longrightarrow S(W_{2n}\otimes X'_{r-r'})\]
is the partial Fourier transform in the subspace $(Y_n^* \otimes
{X'}^*_{r-r'})(mathbb{A})$, and $ev_0$ is  evaluation at $0$. It is clear that
if $r'' < r' < r$, one has
\begin{equation} \label{E:ikeda}
Ik^{n,r', r''} \circ Ik^{n, r, r'}  = Ik^{n, r, r''}. 
\end{equation}

In the special case when  $V$ and $V'$ are complementary
spaces, we shall simply write $Ik^{n,r}$ for $Ik^{n,r,r'}$.  We will call
$Ik^{n,r}$ (more generally $Ik^{n,r,r'}$) an Ikeda map.

\subsection{Weil representation}
Let $\omega_{n,r}$ be the Weil representation of $G_{2n}(\mathbb{A})\times H_r(\mathbb{A})$. More precisely, given a Schwartz-Bruhat function $\phi\in S(Y_n^\ast\otimes V_r )(\mathbb{A})$, the $P(Y_n)(\mathbb{A})\times H_r(\mathbb{A})$-action is given by
\[
\begin{cases}
\omega_{n,r}(1, h)\phi(x)  = \phi(h^{-1} \cdot x),  &\text{  if $h \in H_r(\mathbb{A})$;} \\
\omega_{n,r}(a,1) \phi(x)  =  \chi_V(\nu(a)) \cdot \alpha_E(\nu (a))^{md/2}
\cdot \phi( a^{-1} \cdot x),  &\text{  for $a \in L(Y_n)(\mathbb{A}) =
	\GL(Y_n)(\mathbb{A})$;} \\
\omega_{n,r}(u,1) \phi(x)  = \psi(\frac{1}{2} \cdot  \langle  u(x),
x\rangle)\cdot \phi(x),  &\text{  for $u \in N(Y_n)(\mathbb{A})  \subset
	\Hom(Y_n^*, Y_n)(\mathbb{A})$.} 
\end{cases} \]

\subsection{The Fourier transform $\mathcal{F}_{n,r}$}
There is a partial Fourier transform
\[\mathcal{F}_{n,r}:S(Y_n^\ast \otimes V_r)(\mathbb{A})\longrightarrow S(W_{2n}\otimes X_r^\ast)(\mathbb{A}) \otimes S(Y_n^\ast \otimes V_0)(\mathbb{A}) \]
which is given by integration over the subspace
$(Y_n^\ast\otimes X_r )(\mathbb{A})$. We may regard $\mathcal{F}_{n,r}(\phi)$ as a function on $(W_{2n}\otimes X_r^\ast)(\mathbb{A})$ taking values in $S(Y_n^\ast \otimes V_0)(\mathbb{A})$. 

\section{Eisenstein series}
In this section, we will study the analytic behavior of the Eisenstein series at certain points.
\subsection{The Siegel Eisenstein series}
Let $G_{2n}$ be the unitary group of $W_{2n}$. Let $P(Y_n)$ be the Siegel parabolic subgroup of $G_{2n}$.
Given a normalized induced representation $I(s,\chi_V)=Ind_{P(Y_n)(\mathbb{A})}^{G_{2n}(\mathbb{A})}\chi_V\alpha_E^s$, one can construct an Eisenstein series 
\[E(g;f^{(s)})=\sum_{\gamma\in P(Y_n)(F)\backslash G_{2n}(F) }f^{(s)}(\gamma g) \]
for $f^{(s)}\in I(s,\chi_V)$ and $g\in G_{2n}(\mathbb{A})$. Sometimes we write
\[E(g;f^{(s)})= E^{(n,n)}(g;f^{(s)}) \]
when we want to emphasize the rank of the group. It admits a meromorphic continuation to the whole $s$-plane.

Let $a(s,\chi_V)=\prod_{j=1}^{dn}L(2s-j+1,\omega_{E/F}^{j+d(n+m)})$ and
\[b(s,\chi_V)=\prod_{j=1}^{dn}L(2s+j,\omega_{E/F}^{j+d(n+m)}). \]
\begin{proof}[Proof of Lemma \ref{siegeleisenstein}]
 Suppose that $f^{(s)}\in I(s,\chi_V)$. The normalized intertwining operator $M_n(s,\chi_V)$ in \cite{yamana2013siegel} is given  as follow
	\[M_n(s,\chi_V)f^{(s)}(g)=a(s,\chi_V)^{-1}\int_{N(Y_n)(\mathbb{A})} f^{(s)}(\begin{pmatrix}
	0&\mathbf{1}_n\\-\mathbf{1}_n&0
	\end{pmatrix}ng)dn . \]
	Then $M_n(s,\chi_V)$ is entire due to \cite[Lemma 1.2]{yamana2013siegel}. Moreover, at the point $s=s_0=(m-n)d/2$,
	\[ord_{s=s_0} E(g;f^{(s)})=ord_{s=s_0}\frac{a(s,\chi_V)}{b(s,\chi_V)}=-1. \]
	Therefore, $E(s,f^{(s)})$ has a simple pole at $s=s_0=(m-n)d/2$.
\end{proof}
Given a function $\phi\in S(Y_n^\ast\otimes V_r)(\mathbb{A})$,
set $$f^{(s)}(g)=\Phi^{n,r}(\phi)(g)=\omega_{n,r}(g)\phi(0)$$ and then $f^{(s)}\in I(s_0,\chi_V)$ which is called the Siegel-Weil section. Its image in $I(s_0,\chi_V)$ is isomorphic to the maximal $H_r(\mathbb{A})$-invariant quotient of $\omega_{n,r}$ by \cite[Proposition 1.4]{yamana2013siegel}.
 Let $f^{(s)}=\Phi^{n,r}(\phi)$ be the Siegel-Weil section so that 
\[E(g;\Phi^{n,r}(\phi))=A^{n,r}_{-1}(\phi)(s-s_0)^{-1}+A^{n,r}_0(\phi)+\cdots. \]
Here $A_0^{n,r}(\phi)$ denotes $Val_{s=s_0}E(g;\Phi^{n,r}(g))$.

There are local analogous notation for the intertwining operator and the Siegel-Weil section. Suppose that $V_r(F_v)$ is a Hermitian vector space over $D_v$. The maximal $H_r(F_v)$-invariant quotient of $(\omega_{n,r,v})_{H_r(F_v)}$ is isomorphic to a subrepresentation of $I_v(s_0,\chi_V)$, denoted by $R_n(V_r(F_v))$.

Let $\mathcal{C}=\{\mathfrak{V}_v \}$ be a collection of local Hermitian spaces of dimension $m$ over $D_v$ such that $\mathfrak{V}_v$ is isometric to $V_{r}(F_v)$ for almost all $v$. We form a restricted tensor product $\Pi(\mathcal{C},\chi_V)=\otimes_v' R_{n}(\mathfrak{V}_v)$, which we can regard as a subrepresentation of $I(s_0,\chi_V)$. If there is a global Hermitian $D$-vector space with $\mathfrak{V}_v$ as its completions, then we call $\mathcal{C}$ coherent. Otherwise, we call the collection $\mathcal{C}$ incoherent. By \cite[Proposition 1.4]{yamana2013siegel}, we see that the maximal semisimple quotient of $I(s_0,\chi_V)$ is given by
\[\bigoplus_\mathcal{C} \Pi(\mathcal{C},\chi_V) \]
where the sum runs over all the collections $\mathcal{C}$ (coherent or incoherent) as defined above.

Due to \cite[Proposition 3.5]{yamana2013siegel}, the image of $A^{n,r}_{-1}(\phi)$ is given by
$$\bigoplus_{\mathcal{C} }\Pi(\mathcal{C},\chi_V)  $$
where $\mathcal{C}$ runs over coherent collections.
\begin{prop}\label{secondterm}
	The leading term $A_{-1}^{n,r}(\phi)$ is $G_{2n}(\mathbb{A})$-equivariant and 
	\[A_0^{n,r}(\omega_{n,r}(g)\phi)=g\cdot A_{0}^{n,r}(\phi)\pmod{Im A_{-1}^{n,r}} \]
	for any $g\in G(\mathbb{A})$ and $\phi\in S(Y_n^\ast\otimes V_r)(\mathbb{A})$.
\end{prop}
Note that when $v\in S$, $R_n(V_r(F_v))$ is the full induced representation $I(s_0,\chi_V)$. (See \cite[Proposition 1.4]{yamana2013siegel}.) Then Proposition \ref{secondterm} follows from \cite[Proposition 6.4]{gan2014siegelweil}. 

\subsection{The non-Siegel Eisenstein series} Recall that
\begin{align*} \mathcal{E}^{n,r}(s,\phi)(g)&= \frac{1}{\tau(H_r)\cdot\kappa_r\cdot P_{n,r}(s)}\int_{H_r(F)\backslash H_r(\mathbb{A}) }\Theta(g,h;z.\phi)E_{H_r}(s,\phi)dh
\\&=\sum_{i\geq-2}B^{n,r}_i(\phi)(g)(s-\rho_{H_r})^i. \end{align*}
\begin{lem}
	There exists a function $c_r(s)$ such that
	\[E_{H_r}(s,-)=c_r(s)\cdot E_{H_r}(-s,-). \]
\end{lem}
Unfolding the Eisenstein series $E_{H_r}(s,-)$, one can
obtain the following.
\begin{prop}\cite[Proposition 3.3]{gan2014siegelweil}
	Assume that $Re(s)$ is sufficiently large. Then
	\[\mathcal{E}^{n,r}(s,\phi)=E^{(n,r)}(s,f^{n,r}(s,\pi_{K_{H_r}}(\phi))). \]
	The following explains the notation in the above proposition:
	\begin{itemize}
		\item $E^{(n,r)}$ refers to the Eisenstein series associated to the family of induced representations \[I_r^n(s,\chi_V)=Ind_{Q(Y_r)}^{G_{2n}(\mathbb{A})}(\chi_V\alpha_E^s\boxtimes\Theta_{n-r,0}(V_{0})  ) \]
		where we recall that the Levi factor of $Q(Y_r)$ is $L(Y_r)\cong \GL(Y_r)\times G_{2n-2r}$ and $$\Theta_{n-r,0}(V_0)=\langle \frac{1}{\tau(V_0)}\int_{H_0(F)\backslash H_0(\mathbb{A})}\Theta_{n-r,0}(g,h;\phi)dh:\phi\in S(Y_{n-r}^\ast\otimes V_0)(\mathbb{A}) \rangle. $$
		If $m_0=0$, then $\Theta_{n-r,0}(V_0)$ is interpreted to be the character $\chi_V\circ\iota \circ \nu_{G_{2n-2r}}$ where $\iota:E^\times/F^\times\rightarrow E^1$ is the natural isomorphism and $\nu_{G_{2n-2r}}:G_{2n-2r}\rightarrow E^1$ is the reduced norm map.
		\item $\pi_{K_{H_r}}$ is the projection operator onto the $K_{H_r}$-fixed subspace, defined by
		\[\pi_{K_{H_r}}(\phi)=\int_{K_{H_r}}\omega_{n,r}(k)(\phi)dk.  \]
		\item For $\phi\in S(Y_n^\ast\otimes V_r)(\mathbb{A})$,
		\[f^{n,r}(s,\phi)\in I_r^n(s,\chi_V) \]
is a meromorphic section given by

 \begin{align*}   f^{n,r}(s,\phi)(g)  &= 
\int_{\GL(X_r)(\mathbb{A})} I_{n-r,0}(\omega_{n,r}(g,a)\mathcal{F}_{n,r}(\phi)
( \beta_0)(0 , -) )\,\cdot   \alpha_E(\nu(a))^{s  - \rho_H} \, da \\
&=  \int_{\GL(X_r)(\mathbb{A})} I_{n-r,0}(\omega_{n,r}(g)\mathcal{F}_{n,r}(\phi)
( \beta_0 \circ a )(0 , -) )\,\cdot   \alpha_E(\nu(a))^{s  +nd -
	\rho_H} \, da.
\end{align*}
Here we note that  $\mathcal{F}_{n,r}(\phi)$ is a Schwartz function on
$X_r^* \otimes W_n = \Hom(X_r, W_n)$ taking values in 
\[  \mathcal{S}(Y_n^* \otimes V_0)(\mathbb{A})  =  \mathcal{S}(Y_r^* \otimes V_0)(\mathbb{A})
\otimes \mathcal{S}({Y'}^*_{n-r}\otimes V_0)(\mathbb{A}), \]
and 
\[  \beta_0 \in \Hom(X_r, W_{n}) \]
is defined by
\[  \beta_0(x_i)  = y_i  \quad \text{for $i = 1,\dots,r$,} \]
so that
\[  \mathcal{F}_{n,r}(\phi)(\beta_0 \circ a)(0,-) \in \mathcal{S}({Y'}^*_{n-r} \otimes V_0)(\mathbb{A}). \]
The integral defining $f^{n,r}(s,\phi)$ converges when 
\[  {\rm Re}(s) > \frac{md}{2} - \frac{(2n-r)d }{2} \]
and extends to  a meromorphic section of $I^n_r(s, \chi)$ (since it
is basically a Tate-Godement-Jacquet zeta integral). 
 When $r=0$ and $m_0>0$, we set $f^{n,0}(s,\phi)(g)=I_{n,0}(\phi)(g)$ by convention.
	\end{itemize}
\end{prop}
Following \cite[\S4.2]{gan2014siegelweil}, we express elements of $Y_n^\ast\otimes V_r$ as $3\times 2$ matrices corresponding to the decompositions
\[Y_n^\ast=Y_r^\ast\oplus {Y'}_{n-r}^\ast\mbox{  and  }V_r=X_r\oplus V_0\oplus X_r^\ast, \]
so the first column of the matrix has entries from $Y_r^\ast\otimes X_r,Y_r^\ast\otimes V_0$ and $Y_r^\ast\otimes X_r^\ast$ in this order, and the second column has entries from ${Y'}_{n-r}^\ast\otimes X_r,{Y'}^\ast_{n-r}\otimes V_0$ and ${Y'}_{n-r}^\ast\otimes X_r^\ast$.
\begin{lem}
	\cite[Lemma 4.1]{gan2014siegelweil} One has
	\[f^{n,r}(g)=I_{n-r,0}(\mathfrak{f}^{n,r}(s,\phi)(g) ) \]
	where $\mathfrak{f}^{n,r}(s,\phi)(g)(-)=\int_{\GL(X_r)(\mathbb{A})}\int_{({Y'}_{n-r}^\ast\otimes X_r)(\mathbb{A})}\omega_{n,r}(g)\phi\begin{pmatrix}
	A& X_2\\0&-\\0&0
	\end{pmatrix} \alpha_E(\nu(A))^{-s+rd-dn+\rho_{H_r}}dX_2dA $
\end{lem}
Moreover, one can extend the definition of $f^{n,r}(s,\phi)$ to define functions $F^{n,r}(s,\phi)$ on $G_{2n}\times H_r$ such that $F^{n,r}(s,\phi)\in I_r^n(s,\chi_V)\boxtimes I_{H_r}(-s)$ and $F^{n,r}(s,\phi)|_{G_{2n}}=f^{n,r}(s,\phi)$, see \cite[Remark 4.3]{gan2014siegelweil}.

Now we consider the restriction of the section $f^{n+1,r}(s,\phi)$ from $G_{2n+2}$ to $G_{2n}$ which is closely related to the Ikeda map $Ik^{n,r,r-1}$. More precisely, fix $\phi_1\in S(Y_1^\ast\otimes V_r)(\mathbb{A})$ satisfying:
\begin{itemize}
	\item $\phi_1(0)=1;$
	\item $\phi_1$ is $K_{H_r}$-invariant, so that $\pi_{K_{H_r}}(\phi_1)=\phi_1$.
\end{itemize}
For any $\phi\in S({Y'}_n^\ast\otimes V_r)(\mathbb{A})$, we set 
\[\tilde{\phi}=\phi_1\otimes\phi\in S(Y_{n+1}^\ast\otimes V_r )(\mathbb{A}) .\]
Then $\pi_{K_{H_r}}(\tilde{\phi})=\phi_1\otimes\pi_{K_{H_r}}(\phi)$. Let $W_{2n}=\langle y_2,\cdots,y_{n+1},y_{n+1}^\ast,\cdots,y^\ast_2\rangle \subset W_{2n+2}$ and $$G_{2n}=U(W_{2n})\subset U(W_{2n+2})= G_{2n+2}.$$
\begin{prop}\cite[Proposition 4.2]{gan2014siegelweil} Suppose $m_0>0$ when $r=1$. Then
	there is a constant $\alpha_r>0$ such that
	\[ f^{n+1,r}(s,\pi_{K_{H_r}}(\tilde{\phi}))|_{G_{2n}}=\alpha_r Z_1(-s-(n+1-r)d+\rho_{H_r},\phi_1)\cdot f^{n,r-1}(s+d/2,Ik^{n,r,r-1}(\pi_{K_{H_r}}(\phi))), \]
	where $Z_1(s,\phi_1)$ is the Tate zeta integral
	\[Z_1(s,\phi_1)=\int_{\GL_1(Y_1^\ast)(\mathbb{A}) }\phi_1(ty_1^\ast\otimes x_1)\alpha_E(\nu(t))^sdt. \]
	Moreover, the constant $\alpha_r$ is given in \cite[Lemma 9.1]{ichino2004mathz}.\label{functionf(r)}
\end{prop}
\begin{proof} It suffices to consider the function $\mathfrak{f}^{n,r}(s,\phi)(-)$.
	Assume that $r=1$ and $m_0>0$. Observe that
	\begin{align*}&\mathfrak{f}^{n+1,r}(s,\pi_{K_{H_r}}\tilde{\phi})(g)(-)\\
	&=\int_{({Y'}^\ast_{n+1-r}\otimes X_r)(\mathbb{A}) } \int_{\GL_r(D(\mathbb{A}))}\phi_1(A)\alpha_E(\nu(A))^{-s+rd-dn-d+\rho_{H_r}} dA
	 \omega_{n,r}(g)\pi_{K_{H_r}}(\phi)\begin{pmatrix}
	Y\\-\\0
	\end{pmatrix}dY\\
	&=Z_1(-s-(n+1-r)d+\rho_{H_r},\phi_1)\cdot \mathfrak{f}^{n,r-1}(s+d/2,\phi)(g)(-)  \end{align*}
	because $r-1=0$.
	The proposition holds with $\alpha_1=1$.
	\par
	 If $r>1$, then we use the Iwasawa decomposition on $\GL_r(D(\mathbb{A}))$. Namely, we have
	\[A=k\cdot \begin{pmatrix}
	1&u\\&1
	\end{pmatrix}\cdot \begin{pmatrix}
	t&\\&B
	\end{pmatrix}=k\cdot\begin{pmatrix}
	t& uB\\&B
	\end{pmatrix} \]
	with 
	\begin{itemize}
		\item 	$t\in\GL_1(D(\mathbb{A}))$;
		\item $u\in D(\mathbb{A})^{r-1}$;
		\item $B\cong \GL_{r-1}(D(\mathbb{A}))$;
		\item $k$ is an element in a maximal compact subgroup $K=K_{H_r}\cap \GL_r(D(\mathbb{A}))$ of $\GL_r(D(\mathbb{A}))$.
	\end{itemize}
Accordingly, we have a constant $\alpha_r$ such that
\[\int_{\GL_r(D(\mathbb{A}))}\varphi(A)dA=\alpha_r\cdot \int_{\GL_1(D(\mathbb{A}))}\int_{\GL_{r-1}(D(\mathbb{A}))}\int_{D(\mathbb{A})^{r-1}}\int_K \varphi(k\cdot\begin{pmatrix}
t&uB\\0&B
\end{pmatrix}) dtdBdudk     \]
for any $\varphi\in C_c^\infty(\GL_r(D(\mathbb{A})))$. Moreover, the explicit formula for $\alpha_r$ is given in \cite[Lemma 9.1]{ichino2004mathz}. Since the function $\pi_{K_{H_r}}\tilde{\phi}$ is $K_{H_r}$-invariant, the integral over $dk$ gives the value $1$ and thus disappears.
 Hence
\begin{align*}
&\mathfrak{f}^{n+1,r}(s,\pi_{K_{H_r}}\tilde{\phi})(g)(-)\\
=&\alpha_r\cdot \int_t\int_B\int_u\int_Y \phi_1\otimes\omega_{n,r}(g)\pi_{K_{H_r}}\phi\begin{pmatrix}\begin{matrix}
t\\0
\end{matrix}&
\begin{matrix}
uB\\B
\end{matrix}&Y\\
0&0&-\\0&0&0
\end{pmatrix} \\
&\times  \alpha_E(\nu(t))^{-s-(n+1-r)d+\rho_{H_r}}\alpha_E(\nu(B))^{-s-(n+1-r)d+\rho_{H_r}}dtdBdudY\\
=&\alpha_r\cdot Z_{1}(-s-(n+1-r)d+\rho_{H_r},\phi_1)\\
&\times \int_u\int_B\int_Y \omega_{n,r}(g)\pi_{K_{H_r}}\phi \begin{pmatrix}
\begin{matrix}
uB\\ B
\end{matrix}&Y\\ 0&-\\0&0
\end{pmatrix}\alpha_E(\nu(B))^{-s-(n+1-r)d+\rho_{H_r}}dYdBdu\\
=&\alpha_r\cdot Z_{1}(-s-(n+1-r)d+\rho_{H_r},\phi_1)\\
&\times \int_B\int_{Y_2}\int_{Y_1}\int_{u}  \omega_{n,r}(g)\pi_{K_{H_r}}\phi \begin{pmatrix}
\begin{matrix}
u\\ B
\end{matrix}&\begin{matrix}
Y_1\\Y_2
\end{matrix}\\ 0&-\\0&0
\end{pmatrix}\alpha_E(\nu(B))^{-s-d-(n+1-r)d+\rho_{H_r}}dudY_1dY_2dB\\
=&\alpha_r Z_1(-s-(n+1-r)d+\rho_{H_r},\phi_1)\mathfrak{f}^{n,r-1}(s+d/2,Ik^{n,r,r-1}(\pi_{K_{H_r}}\phi))(g)(-)
\end{align*}
since $\rho_{H_r}=\rho_{H_{r-1}}+d/2$.
This finishes the proof of Proposition \ref{functionf(r)}.

\end{proof}

\section{The Siegel-Weil formula}
Let $V_{r'}$ be the complementary space of $V_r$. 
Suppose that $0<m'=m_0+2r'\leq n$ with $r'>0$.
\begin{thm}
	\cite[Theorem 2]{yamana2013siegel} Fix a function $\phi'\in S(Y_n^\ast\otimes V_{r'})$. Let $f'$ be the Siegel-Weil section associated to $V_{r'}$. Then the Siegel Eisenstein series $E(s,f')$ is holomorphic at $s=(m'-n)d/2 $ and $A_0^{n,r'}(\phi')=2B_{-1}^{n,r'}(\phi')$.
	In particular, if $m=n$ so that $r=r'$, then $A_{0}^{n,r}(\phi)=2B_{-1}^{n,r}(\phi)$ for $\phi\in S(Y_n^\ast\otimes V_r)(\mathbb{A})$.
\end{thm}
This is called the regularized Siegel-Weil formula in the first term range. There is another form:
\[A_{-1}^{n,r}(\phi)=\kappa_{r,r'} B_{-1}^{n,r'}(Ik^{n,r}\pi_{K_{H_r}}\phi) \]
for any $\phi\in S(Y_n^\ast\otimes V_r)(\mathbb{A})$ due to \cite[Theorem 4.1]{ichino2004mathz}. In particular, $A_{-1}^{n,r}(\phi)=\kappa_{r,r'}B^{n,r' }_0(Ik^{n,r}\pi_{K_{H_r}}\phi)$ when $r'=0$.
\begin{thm}
[Weil] Let $U(V_0)$ be the anisotropic unitary group defined over $F$. For $\phi\in S(Y_n^\ast\otimes V_0)(\mathbb{A})$, there exists a constant $c>0$ such that
\[A_0^{n,0}(\phi)=c\cdot I_{n,0}(\phi) \]
\end{thm}
\begin{lem}
	\cite[Proposition 7.2]{gan2014siegelweil}
	For $\phi\in S(Y_n^\ast\otimes V_0)(\mathbb{A})=S(y_1^\ast\otimes V_0)(\mathbb{A})\otimes S({Y'}^\ast_{n-1}\otimes V_0)(\mathbb{A})$, we have
	\[I_{n,0}(\phi)_{U_1(Y_1)}|_{\GL(Y_1)(\mathbb{A})\times G_{2n-2}(\mathbb{A}) }=\chi_V\cdot\alpha_E^{m_0d}\boxtimes I_{n-1,0}(\phi(0,-)) \]
	where $I_{n,0}(\phi)_{U_1}$ is the constant term of $I_{n,0}(\phi)$ with respect to the maximal parabolic $Q_1(Y_1)$.
\end{lem}

\begin{proof}[Proof of Theorem \ref{secondtermsiegelweil}]
	Suppose that we are dealing with the Weil representation of $G_{2n+2}\times H_r$ with $m=n+1$. Then for $\tilde{\phi}\in S(Y_{n+1}^\ast\otimes V_r )(\mathbb{A})$, \cite[Theorem 2]{yamana2013siegel} implies that \[A_0^{n+1,r}(\tilde{\phi})=2B_{-1}^{n+1,r}(\tilde{\phi}). \]
	Let us take the constant term of both sides with respect to the maximal parabolic $Q^{n+1}(Y_1)=L^{n+1}(Y_1)\cdot U^{n+1}(Y_1)$ of $G_{2n+2}$, which gives
	 \[  A^{n+1,r}_0(\tilde{\phi})_{U^{n+1}(Y_1)}  = 2 \cdot
	B^{n+1,r}_{-1}(\tilde{\phi})_{U^{n+1}(Y_1)}, \]
	which is an identity of automorphic forms on $L(Y_1) = \GL(Y_1)
	\times G_{2n}$, where $W_{2n} = Y'_n \oplus {Y'_n}^*$. (Note that the
	superscript $^{n+1}$ in the groups $Q^{n+1}(Y_1)$ etc indicates the
	rank of the ambient group $G_{2n+2}$.)
	
	Let $f_s$ be the standard section of
	\[I_r^n(s,\chi_V)=Ind_{Q^{n+1}(Y_r)(\mathbb{A})}^{G_{2n+2}(\mathbb{A})}\chi_V\alpha_E^s\boxtimes\Theta_{n+1-r,0}(V_0). \]
	Let $E^{(n+1,r)}(s,f_s)(g)$ be the associated Eisenstein series, i.e.
	\[E^{(n+1,r)}(s,f_s)(g)=\sum_{\gamma\in Q^{n+1}(Y_r)(F)\backslash G_{2n+2}(F) }f_s(\gamma g)  \]
	for $g\in G_{2n+2}(\mathbb{A})$ and $Re(s)$ sufficiently large. Note that 
	\[\mathcal{E}^{n,r}(s,\phi)=E^{(n,r)}(s,f^{n,r}(s,\pi_{K_{H_r}}(\phi))) \]
	and $A_0^{n+1,r}(\tilde{\phi})_{U(Y_1)}=Val_{s=0}E^{(n+1,n+1)}(s,\Phi^{n+1,r}(\tilde{\phi}))_{U(Y_1)}$.
	So we are interested in computing the constant term $E^{(n+1,r)}(s,f_s)_{U^{n+1}(Y_1)}$. 

	Let us choose the double coset representatives $1,\omega^+$ and $\omega^-$ for the double coset space $Q^{n+1}(Y_r)\backslash G_{2n+2}/Q^{n+1}(Y_1)$, where
	\[\omega^+=\begin{pmatrix}
	J_{r+1}&0\\0&J_{r+1}
	\end{pmatrix} \]
	with $J_{r+1}=\begin{pmatrix}
	0&0&1&0\\0&\mathbf{1}_{r-1}&0&0\\1&0&0&0\\0&0&0&\mathbf{1}_{n-r}
	\end{pmatrix}$
	and \[\omega^-=\begin{pmatrix}
	0&0&1\\0&\mathbf{1}_{2n}&0\\-1&0&0
	\end{pmatrix}. \]
	Associated to the Weyl group element $\omega=\omega^+$ or $\omega^-$ is the standard intertwining operator $M(\omega,s)$:
	\[M(\omega,s)(f_s)(g)= \int_{(U^{n+1}(Y_1)(F)  \cap w Q^{n+1}(Y_r)(F) w^{-1})\backslash U^{n+1}(Y_1)(\mathbb{A})}
	f_s(w^{-1} u  g) \, du. \]
	By the same computation as in \cite[Lemma 8.2]{gan2014siegelweil}, as the automorphic forms on $L^{n+1}(Y_1)=\GL_1(Y_1)\times G_{2n}$,
	\begin{align*}
	E^{(n+1,r)}(s,f_s)_{U^{n+1}(Y_1)}= &\chi_V\alpha_E^{s+(n+1)d-rd/2}E^{(n,r-1)}(s+d/2,f_s|_{G_{2n}})+\chi_V\alpha_E^{md/2}E^{(n,r)}(s,M(\omega^+,s)(f_s)|_{G_{2n}})\\
	&+\chi_V\alpha_E^{-s+nd+d-rd/2}E^{(n,r-1)}(s-d/2,M(\omega^-,s)(f_s)|_{G_{2n}})
	\end{align*}
	and  $E^{(n+1,n+1)}(s,f)_{U^{n+1}(Y_1)}$
	\[=\chi_V\alpha_E^{s+(n+1)d/2}E^{(n,n)}(s+d/2,f|_{G_{2n}})+
	\chi_V\alpha_E^{-s+(n+1)d/2}E^{(n,n)}(s-d/2,M(\omega^-,s)(f)|_{G_{2n}}) \]
	for $f\in I_{n+1}^{n+1}(s,\chi_V)$.
	
	 Choose once and for all $\phi_1\in \mathcal{S}(Y_1^* \otimes
	V_r)(\mathbb{A})$ satisfying:
	
	\begin{itemize}
		\item $\phi_1(0) = 1$;
		\item $\phi_1$ is $K_{H_r}$-invariant, so that $\pi_{K_{H_r}} \phi_1 = \phi_1$. 
	\end{itemize}
	
	Let $Y_n' = \langle y_2, ...,y_{n+1} \rangle$ so that ${Y'_n}^* = \langle y_2^*,...,y_{n+1}^* \rangle$.
	For any $\phi \in \mathcal{S}({Y_n'}^* \otimes V_r)(\mathbb{A})$, we set
	\[  \tilde{\phi} := \phi_1 \otimes \phi \in \mathcal{S}(Y_{n+1}^*
	\otimes V_r)(\mathbb{A}). \]
	Then 
	\[  \pi_{K_{H_r}}(\tilde{\phi}) = \phi_1 \otimes \pi_{K_{H_r}}
	\phi. \]
	Note that the group $G_{2n}$ acts trivially on $\phi_1$, i.e. for
	$g\in G_{2n}(\mathbb{A})$,
	$$\omega_{n+1,r}(g)\tilde{\phi}=\phi_1\otimes\omega_{n,r}(g)\phi.$$ 
	\begin{enumerate}[(i)]
		\item We focus on the second term identity first. Observe that
		for $g\in G_{2n}(\mathbb{A})$, 
		\[\Phi^{n+1,r}(\tilde{\phi})(g)=\phi_1(0)\cdot\omega_{n,r}(g)\phi(0)=\Phi^{n,r}(\phi)(g). \]
		Thus,
		\[E^{(n,n)}(s+d/2,\Phi^{n+1,r}(\tilde{\phi})|_{G_{2n}})=E^{(n,n)}(s+d/2,\Phi^{n,r}(\phi)) \]
		Note that the functional equation implies that
		\[E^{(n,n)}(s-d/2,M(\omega^-,s)(\Phi^{n+1,r}(\tilde{\phi})|_{G_{2n}})=E^{(n,n)}(d/2-s,M_n(s-d/2,\chi_V)(M(\omega^-,s)(\Phi^{n+1,r}(\tilde{\phi}))|_{G_{2n}})). \]
	where $M_n(s,\chi_V)$ is the normalized intertwining operator for the Siegel principal series. By the result of Kudla-Rallis in \cite[Lemma 1.2.2]{kudlarallis}, 
		\[M_n(s-d/2,\chi_V)M(\omega^-,s)=M_{n+1}(s,\chi_V) \]
		which is holomorphic at $s=0$.
		Moreover, $M_{n+1}(0,\chi_V)\Phi^{n+1,r}\tilde{\phi}=\Phi^{n+1,r}(\tilde{\phi})$. So
	by a similar computation appearing in \cite[\S9.2]{gan2014siegelweil}, one has
		\[A^{n+1,r}(\tilde{\phi})_{U^{n+1}(Y_1)}= 2A_0^{n,r}(\phi)\pmod{Im(A_{-1}^{n,r})} \]
		as the automorphic forms on $ G_{2n}$.
		
		Since $$B_{-1}^{n+1,r}(\tilde{\phi})=Res_{s=\rho_{H_r}}\mathcal{E}^{n+1,r}(s,\tilde{\phi})=Res_{s=\rho_{H_r}}E^{(n+1,r)}(s,f^{n+1,r}(s,\pi_{K_{H_r}}(\tilde{\phi}))),$$ $B_{-1}^{n+1,r}(\tilde{\phi})_{U^{n+1}(Y_1)}$ is the residue at $s=\rho_{H_r}=(m-r)d/2$ of the function
		\begin{align*}\chi_V\alpha_E^{s+(n+1)d-rd/2}E^{(n,r-1)}(s+d/2,f^{n+1,r}(s,\pi_{K_{H_r}}(\tilde{\phi}))|_{G_{2n}})+\chi_V\alpha_E^{md/2}E^{(n,r)}(s,M(\omega^+,s)(f^{n+1,r}(s,\pi_{K_{H_r}}(\tilde{\phi})))|_{G_{2n}})\\
		+\chi_V\alpha_E^{-s+nd+d-rd/2}E^{(n,r-1)}(s-d/2,M(\omega^-,s)(f^{n+1,r}(s,\pi_{K_{H_r}}(\tilde{\phi})))|_{G_{2n}}). \end{align*}
		Note that $m=n+1$, so that $r'=r-1$. Then Proposition \ref{functionf(r)} implies
\[f^{n+1,r}(s,\pi_{K_{H_r}}(\tilde{\phi}))|_{G_{2n}}=\alpha_r Z_1(-s-\rho_{H_r},\phi_1) f^{n,r-1}(s+d/2,Ik^{n,r}(\pi_{K_{H_r}}(\phi))). \]
We will mainly concern the $\chi_V\alpha_E^{md/2}$-part of the residue at $s=\rho_{H_r}$ of
\[E^{(n+1,r)}(s,f^{n+1,r}(s,\pi_{K_{H_r}}(\tilde{\phi})))_{U^{n+1}(Y_1)}. \]
 Due to \cite[Lemma 9.1]{gan2014siegelweil},  $E^{(n,r-1)}(s+d/2,f^{n+1,r}(s,\pi_{K_{H_r}}(\tilde{\phi}))|_{G_{2n}})$ is holomorphic at $s=\rho_{H_r}$. Thanks to \cite[Proposition 9.2]{gan2014siegelweil},
\[M(\omega^+,s)(f^{n+1,r}(s,\pi_{K_{H_r}}(\tilde{\phi}))|_{G_{2n}})=f^{n,r}(s,\pi_{K_{H_r}}(\phi)) \]
which implies that \[E^{(n,r)}(s,M(\omega^+,s)(f^{n+1,r}(s,\pi_{K_{H_r}}(\tilde{\phi}))|_{G_{2n}}))=\mathcal{E}^{n,r}(s,\phi).  \]
It has a residue $B_{-1}^{n,r}(\phi)$ at $s=\rho_{H_r}$.

For the last term, the functional equation implies that
\begin{align*}
&E^{(n,r-1)}(s-d/2,M(\omega^-,s)(f^{n+1,r}(s,\pi_{K_{H_r}}(\tilde{\phi}))|_{G_{2n}}))
\\=&E^{(n,r-1)}(d/2-s,M_n(\omega_{r-1},s-d/2)(M(\omega^-,s)f^{n+1,r}(s,\pi_{K_{H_r}}(\tilde{\phi}))|_{G_{2n}}))\\
=&E^{(n,r-1)}(d/2-s,M_{n+1}(\omega_r,s)f^{n+1,r}(s,\pi_{K_{H_r}}(\tilde{\phi}))|_{G_{2n}})\\
=&c_r(s)\cdot E^{(n,r-1)}(d/2-s,f^{n+1,s}(-s,\pi_{K_{H_r}}(\tilde{\phi}))|_{G_{2n}})\\
=&c_r(s)\cdot \alpha_r Z_1(s-\rho_{H_r},\phi_1)\cdot E^{(n,r-1)}(d/2-s,f^{n,r-1}(-s+d/2,Ik^{n,r}(\pi_{K_{H_r}}(\phi))))\\
=&c_r(s)\cdot\alpha_r Z_1(s-\rho_{H_r},\phi_1)\mathcal{E}^{n,r-1}(d/2-s,Ik^{n,r}(\pi_{K_{H_r}}(\phi)))\\
=&\frac{c_r(s)}{c_{r-1}(s-d/2)}\alpha_r Z_1(s-\rho_{H_r},\phi_1)\mathcal{E}^{n,r-1}(s-d/2,Ik^{n,r}(\pi_{K_{H_r}}(\phi)))
\end{align*}
due to \cite[Lemma 1.2.2]{kudlarallis} and \cite[Remark 9.4]{gan2014siegelweil}, where $$\omega_{r-1}=\begin{pmatrix}
0&0&\mathbf{1}_{r-1}\\0&\mathbf{1}_{2n+2-2r}&0\\-\mathbf{1}_{r-1}&0&0
\end{pmatrix},\quad \omega_r=\begin{pmatrix}
0&0&\mathbf{1}_r\\0&\mathbf{1}_{2n+2-2r}&0\\-\mathbf{1}_r&0&0
\end{pmatrix}$$ and $c_r(s)$ is the meromorphic function satisfying $$E_{H_r}(s,-)=c_r(s)E_{H_r}(-s,-).$$
Note that
\begin{itemize}
	\item $c_r(s)$ has a simple pole at $s=\rho_{H_r}=\rho_{H_{r-1}}+d/2$, then
	$$\frac{c_r(s)}{c_{r-1}(s-d/2)}$$ is holomorphic and nonzero at $s=\rho_{H_r}$ when $r>1$;
	\item the Tate zeta integral $Z_1(s-\rho_{H_r},\phi_1)$ has a simple pole at $s=\rho_{H_r}$;
	\item $$\mathcal{E}^{n,r-1}(s-d/2,Ik^{n,r}(\pi_{K_{H_r}}\phi))=\sum_{i\geq-1}B_{i}^{n,r-1}(Ik^{n,r}\pi_{K_{H_r}}\phi)(s-\rho_{H_{r-1}}-d/2)^i $$ and $B_{-1}^{n,r-1}=0$ if $r=1$.
\end{itemize}
Taking the residue at $s=\rho_{H_r}=\rho_{H_{r-1}}+d/2$, we have
\begin{align*}
A_0^{n,r}(\phi)-B_{-1}^{n,r}(\phi)=a_1B_{-1}^{n,r'}(Ik^{n,r}\pi_{K_{H_r}}\phi)+a_2 B_0^{n,r'}(Ik^{n,r}(\pi_{K_{H_r}}\phi))\pmod{Im A_{-1}^{n,r}}
\end{align*} for some constants $a_1,a_2$.
By the first term identity in the first term range, $$B_{-1}^{n,r'}(Ik^{n,r}\pi_{K_{H_r}}\phi)\in Im(A_{-1}^{n,r}).$$ Then we get the desired identity when $m=n+1$. If $r=1$, then $r'=0$ and
\[B_{-1}^{n,0}(Ik^{n,1}\pi_{K_{H_r}}\phi)\in Im(A_{-1}^{n,1}). \]
		\item Let us focus on the first term identity now.
		In fact,  the last term $$E^{(n,r-1)}(s-d/2,M(\omega^-,s)f^{n+1,r}(s,\pi_{K_{H_r}}\tilde{\phi})|_{G_{2n}})$$
		has a pole of second order at $s=\rho_{H_r}$. It has a leading term
		\begin{equation}B_{-1}^{n,r-1}(Ik^{n,r}\pi_{K_{H_r}}(\phi))\cdot\alpha_r Val_{s=\rho_{H_r}}\frac{c_r(s)}{c_{r-1}(s-d/2)} \cdot Res_{s=\rho_{H_r}}Z_1(s-\rho_{H_r},\phi_1) \label{leadingterm} \end{equation}
		when $r>1$ and $Res_{s=\rho_{H_r}}Z_1(s-\rho_{H_r},\phi_1)$ only depends on the division algebra $D$.
		The leading term (\ref{leadingterm}) must be cancelled with the leading term $B_{-2}^{n,r}(\phi)$ of $\mathcal{E}^{n,r}(s,\phi)$. Moreover
		$$B_{-1}^{n,r-1}(Ik^{n,r}\pi_{K_{H_r}}\phi)=\kappa_{r,r'}\cdot A_{-1}^{n,r}(\phi)$$ by the first term identity in the first term range.
	Hence there exists a constant $c>0$ such that $A_{-1}^{n,r}(\phi)=c \cdot B_{-2}^{n,r}(\phi)$.
	If $r=1$, then $B_{-1}^{n,r-1}=0$ and $\frac{c_r(s)}{c_{r-1}(s-d/2)}$ has a pole at $s=\rho_{H_r}$. This finishes the proof when $m=n+1$.
	\end{enumerate}
	In general, if $n+1< m\leq n+r$, we may assume that 
	\[A^{n+1,r}_{-1}(\tilde{\phi})=c\cdot B_{-2}^{n+1,r}(\tilde{\phi}) \]
	and 
	\begin{equation}
\label{secondterminduction}	
	A_0^{n+1,r}(\tilde{\phi})= B_{-1}^{n+1,r}(\tilde{\phi})+c'\cdot B_{0}^{n+1,r'}(Ik^{n+1,r}(\pi_{K_{H_r}}\tilde{\phi}))+A_{-1}^{n+1,r}(\varphi) \end{equation}
	for some $\varphi\in S(Y_{n+1}^\ast\otimes V_r)(\mathbb{A})$, where $m_0+r+r'=n+1$ and $r'\geq1$.
	
	We still consider the constant term along $U^{n+1}(Y_1)$ and get
	\[A_{-1}^{n+1,r}(\tilde{\phi})_{U^{n+1}(Y_1)}=c\cdot B_{-2}^{n+1,r}(\tilde{\phi})_{U^{n+1}(Y_1)}. \]
	We concern the terms in
	\[E^{(n+1,r)}(s,f^{n+1,r}(s,\pi_{K_{H_r}}\tilde{\phi}))_{U^{n+1}(Y_1)} \]
	where $\GL_1(Y_1)\subset L(Y_1)$ acts by the character $\chi\cdot \alpha_E^{md/2}$. Then we have
	\[\chi_V\alpha_E^{md/2}\cdot E^{(n,r)}(s,M(\omega^+,s)(f^{n+1,r}(s,\pi_{K_{H_r}}\tilde{\phi}))|_{G_{2n}})=\chi_V\alpha_E^{md/2}\mathcal{E}^{n,r}(s,\phi) \]
	and so the $\chi_V\alpha_E^{md/2}$-part of $B_{-2}^{n+1,r}(\tilde{\phi})_{U^{n+1}(Y_1)}$ equals to
	$B_{-2}^{n,r}(\phi) $. On the other hand, the $\chi_V\alpha_E^{md/2}$-part of $A_{-1}^{n+1,r}(\tilde{\phi})_{U^{n+1}(Y_1)}$ is the residue at $s=(m-1-n)d/2$ of
\[E^{(n,n)}(s+d/2,\Phi^{n+1,r}(\tilde{\phi})|_{G_{2n}})=E^{(n,n)}(s+d/2,\Phi^{n,r}(\phi)), \]
which is nothing but $A_{-1}^{n,r}(\phi)$. Thus there exists a constant $c$ such that
\[A_{-1}^{n,r}(\phi)=(\mbox{the }\chi_V\alpha_E^{md/2}\mbox{-part of } A_{-1}^{n+1,r}(\tilde{\phi})_{U^{n+1}(Y_1)})=c\cdot B_{-2}^{n,r}(\phi). \]
Observe that
\[A_{-1}^{n+1,r}(\tilde{\phi})_{U^{n+1}(Y_1)}=Res_{s=(m-n-1)d/2}E^{(n,n)}(s,\Phi^{n+1,r}(\tilde{\phi}))_{U^{n+1}(Y_1)} \]
and so the $\chi_V\alpha_E^{md/2}$-part of $A_{-1}^{n+1,r}(\tilde{\phi})_{U^{n+1}(Y_1)}$ lies in $Im A_{-1}^{n,r}$.
Similarly, we compute the constant term along $U^{n+1}(Y_1)$ of both sides of (\ref{secondterminduction}) and then extract the terms with $\GL(Y_1)$ acting via $\chi\cdot\alpha_E^{md/2}$. Therefore,
\[A_0^{n,r}(\phi)-B_{-1}^{n,r}(\phi)=c'\cdot (\mbox{the }\chi_V\alpha_E^{md/2}\mbox{-part of }B_0^{n+1,r'}(Ik^{n+1,r}\pi_{K_{H_r}}\tilde{\phi})_{U^{n+1}(Y_1)})\pmod{ Im A_{-1}^{n,r}}. \]
By the definition, $B_0^{n+1,r'}(Ik^{n+1,r}\pi_{K_{H_r}}\tilde{\phi})_{U^{n+1}(Y_1)}$ is the value taking at $s=\rho_{H_{r'}}$ of the function
	\begin{align*}\chi_V\alpha_E^{s+(n+1)d-r'd/2}E^{(n,r'-1)}(s+d/2,\cdots)+\chi_V\alpha_E^{m'd/2}E^{(n,r')}(s,M(\omega^+,s)(\cdots))\\
+\chi_V\alpha_E^{-s+nd+d-r'd/2}E^{(n,r'-1)}(s-d/2,M(\omega^-,s)(\cdots)). \end{align*}

The remaining part of the proof is to show that there exists a nonzero constant $c'$ such that
\[Val_{s=\rho_{H_{r'}}}E^{(n,r'-1)}(s-d/2,M(\omega^-,s)f^{n+1,r'}(s,Ik^{n+1,r}(\pi_{K_{H_r}}\tilde{\phi}))|_{G_{2n}})=c'B_0^{n,r'-1}(Ik^{n,r}\pi_{K_{H_r}}\phi)\pmod{Im A_{-1}^{n,r}} \]
since $r'-1+r+m_0=n$.
Note that
\[Ik^{n+1,r}(\pi_{K_{H_r}}\tilde{\phi})=Ik^{1,r,r'}(\phi_1)\otimes Ik^{n,r,r'}(\pi_{K_{H_r}}\phi). \]
Thus
\begin{align*}
&E^{(n,r'-1)}(s-d/2,M(\omega^-,s)f^{n+1,r'}(s,Ik^{n+1,r}(\pi_{K_{H_r}}\tilde{\phi}))|_{G_{2n}})\\
=&c_{r'}(s)E^{(n,r'-1)}(d/2-s,f^{n+1,r'}(-s,Ik^{n+1,r}\pi_{K_{H_r}}\tilde{\phi})|_{G_{2n}})\\
=&c_{r'}(s)\alpha_{r'} Z_1(s-(n+1-r')d+\rho_{H_{r'}},Ik^{1,r,r'}\phi_1)\\
&\times E^{n,r'-1}(d/2-s, f^{n,r'-1}(-s+d/2,Ik^{n,r',r'-1}\circ Ik^{n,r,r'}\pi_{K_{H_r}}\phi))\\
=&c_{r'}(s)\alpha_{r'} Z_1(s-(n+1-r')d+\rho_{H_{r'}},Ik^{1,r,r'}\phi_1)\\
&\times E^{n,r'-1}(d/2-s,f^{n,r'-1}(-s+d/2,Ik^{n,r,r'-1}\pi_{K_{H_r}}\phi))\\
=&c_{r'}(s)\alpha_{r'} Z_1(s-(n+1-r')d+\rho_{H_{r'}},Ik^{1,r,r'}\phi_1)\cdot \mathcal{E}^{n,r'-1}(d/2-s,Ik^{n,r}\pi_{K_{H_r}}\phi)\\
=&\frac{c_{r'}(s)}{c_{r'-1}(s-d/2)}\alpha_{r'} Z_1(s-(n+1-r')d+\rho_{H_{r'}},Ik^{1,r,r'}\phi_1)\mathcal{E}^{n,r'-1}(s-d/2,Ik^{n,r}\pi_{K_{H_r}}\phi)
\end{align*}
where $V_{r'-1}$ and $V_{r}$ are complementary with respect to $W_{2n}$ and so $Ik^{n,r,r'-1}=Ik^{n,r}$.

If $r'>1$, then both $\frac{c_{r'}(s)}{c_{r'-1}(s-d/2)}$ and $Z_1(s-(n+1-r')d+\rho_{H_{r'}},Ik^{1,r,r'}\phi_1)$ are holomorphic at $s=\rho_{H_{r'}}$. Then 
\[ Val_{s=\rho_{H_{r'}}}E^{(n,r'-1)}(s-d/2,M(\omega^-,s)f^{n+1,r'}(s,Ik^{n+1,r}(\pi_{K_{H_r}}\tilde{\phi}))|_{G_{2n}})=c'B_0^{n,r'-1}(Ik^{n,r}\pi_{K_{H_r}}\phi)\]
for some constant $c'$. If $r'=1$, then $m=r+n$ and $B_0^{n,0}(Ik^{n,r}\pi_{K_{H_r}}\phi)\in Im A_{-1}^{n,r}$.
\end{proof}

\section{Applications to the Rallis inner product formula}
In this section, we use the the regularized Siegel-Weil formula to derive the Rallis inner product formula and prove the non-vanishing theorem of global thetal lifts. Yamana \cite{yamana2014theta} has studied the relation between the nonvanishing of theta lift
and the analytic property of its $L$-fucntion in the first term range, i.e. $m\leq n$. We will focus on the second term range.

Suppose that $E_v=F_v\oplus F_v$ for all archimedean places $v|\infty$. Let $W$ be a skew-Hermitian $D$-vector space and $W_{2n}=W\oplus W^-$, where $W^-$ is the space $W$ with the form scaled by $-1$. Let $V_r$ be the Hermitian $D$-vector space with Witt index $r$ as defined before. Suppose that $W\otimes V_r$ has a complete polarization
\[W\otimes V_r=\mathcal{X}\oplus\mathcal{Y}. \]
Let $\omega_\psi$ be the Weil representation of $U(W)\times H_r$ associated to $W\otimes V_r$. Given a function $\phi\in S(\mathcal{X})(\mathbb{A})$, one can define
\[\theta(\phi)(g,h)=\sum_{x\in\mathcal{X}(F)}\omega_\psi(g,h) \phi(x) \]
for $(g,h)\in U(W)(\mathbb{A})\times H_r(\mathbb{A})$. For a cuspidal representation $\pi$ of $U(W)$, we consider its global theta lift $\Theta_{n,r}(\pi)$ to $H_r$, so that $\Theta_{n,r}(\pi)$ is hte automorphic subrepresentation of $H_r$ spanned by the automorphic forms
\[\theta_{n,r}(\phi,f)(h)=\int_{U(W)(F)\backslash U(W)(\mathbb{A}) }\theta(\phi)(g,h)\cdot \overline{f(g)}dg  \]
for $f\in\pi$.

We will use the doubling see-saw diagram 
\[\xymatrix{G_{2n}\ar@{-}[d]\ar@{-}[rd]& H_r\times H_r\ar@{-}[d]\\ U(W)\times U(W^-)\ar@{-}[ru] & H_r^{\triangle} } \]
to study the inner product
\[\langle \theta_{n,r}(\phi_1,f_1),\theta_{n,r}(\phi_2,f_2) \rangle \]
for $\phi_i\in\omega_\psi$ and $f_i\in\pi$. Indeed, we choose a Witt decomposition of $W_{2n}$ to be
\[W_{2n}=Y_n\oplus Y_n^\ast \]
with $Y_n=W^\triangle=\{(y,y):y\in W \}$ and $Y_n^\ast=\{(y,-y):y\in W \}$. The Weil representation $\omega_{n,r}$ of $G_{2n}\times H_r$ can be realized on $S(Y_n^\ast\otimes V_r)$ such that $H_{r}^\triangle$ acts by
\[\omega_{n,r}(h)\phi(x)=\phi(h^{-1}\cdot x) \]
for $h\in H_r^\triangle$.
Moreover,
\[\omega_{n,r}|_{U(W)\times U(W)}\cong \omega_\psi\otimes(\omega_\psi^\vee\cdot\chi_V)|_{U(W)\times U(W) }. \]
There exists an isomorphism
\[\delta:\omega_\psi\otimes(\omega_\psi^\vee\cdot\chi_V)\longrightarrow \omega_{n,r} \]
such that $\delta(\phi_1\otimes\overline{\phi_2})(0)=\langle \phi_1,\phi_2 \rangle$
for $\phi_i\in S(\mathcal{X})(\mathcal{A})$.

\begin{thm}\label{fininaltheoremRallis}
	Assume that $1+n\leq m\leq n+r$ and $W$ is a skew-Hermitian $D$-vector space of dimension $n$. Let $\pi$ be an irreducible cuspidal representation of $U(W)$ and consider its global theta lift
	$\Theta_{n,r}(\pi) $ to $U(V_r)=H_r$. Assume that $\Theta_{n,j}(\pi)=0$ for $j< r$, so that $\Theta_{n,r}(\pi) $ is cuspidal. Then $\Theta_{n,r}(\pi)$ is nonzero if and only if 
	\begin{enumerate}[(i)]
		\item for all places $v$, $\Theta_{n,r}(\pi_v)\neq0$ and
		\item $L(s_0+1/2,\pi\times\chi_V)\neq0$ where $s_0=(m-n)d/2$. 
	\end{enumerate}
\end{thm}
\begin{proof}
	Let us consider the integral
	\begin{equation}\label{innerprod}
		\int_{H_r(F)\backslash H_r(\mathbb{A}) }\theta_{n,r}(\phi_1,f_1)(h)\overline{\theta_{n,r}(\phi_2,f_2)(h)}E_{H_r}(s,h)dh. 
	\end{equation}
 By the same computation appearing in \cite{kudla1994annals}, the integral (\ref{innerprod}) equals to
 \[\int_{[U(W)\times U(W)]}f_1(g)\overline{f_2(g)}\mathcal{E}^{n,r}(s,\delta(\phi_1\otimes\overline{\phi_2}))((g_1,g_2))\chi_V^{-1}(\nu(g_1))dg_1dg_2 . \]
 Here \[[U(W)\times U(W) ]=(U(W)\times U(W))(F)\backslash (U(W)\times U(W) )(\mathbb{A}) . \]
 The Eisenstein series $E_{H_r}(s,h)$ has a simple pole at $s=\rho_{H_r}$ with a constant residue. Thus
 \[\langle \theta_{n,r}(\phi_1,f_1),\theta_{n,r}(\phi_2,f_2) \rangle=c\cdot\int_{[U(W)\times U(W)]}f_1(g_1)\overline{f_2(g_2)} B_{-1}^{n,r}(\delta(\phi_1\otimes \overline{\phi_2}))((g_1,g_2))\chi_V^{-1}(\nu(g_2)) dg_1dg_2 \]
 for a nonzero constant $c$. Note that 
 \begin{align*}
 &\int_{[U(W)\times U(W)]} f_1(g_1)\overline{f_2(g_2)}A_{-1}^{n,r}(\delta(\phi_1\otimes \overline{\phi_2})(g_1,g_2)\chi_V^{-1}(\nu(g_2))dg_1dg_2\\
 =&\int_{[U(W)\times U(W)]}f_1(g_1)
\overline{f_2(g_2)}B_{-1}^{n,r'}(Ik^{n,r}\pi_{K_{H_r}}\delta(\phi_1\otimes\overline{\phi_2} ))(g_1,g_2)\chi_V^{-1}(\nu(g_2)) dg_1dg_2 \\
=& 0 \end{align*}
since $\theta_{n,r'}(-,f)=0$ for any $f\in \pi$. Similarly,
\[\int_{[U(W)\times U(W)]} f_1(g_1)\overline{f_2(g_2)} B_0^{n,r'}(Ik^{n,r}\pi_{K_{H_r}}\delta(\phi_1\otimes\overline{\phi_2}))(g_1,g_2)\chi_V^{-1}(\nu(g_2)) dg_1dg_2=0. \]
The second term identity in the second term range implies that
\[\langle \theta_{n,r}(\phi_1,f_1),\theta_{n,r}(\phi_2,f_2) \rangle=c\cdot \int_{[U(W)\times U(W)]} f_1(g_1)\overline{f_2(g_2)} A_0^{n,r}(\delta(\phi_1\otimes\overline{\phi_2} ))(g_1,g_2)\chi_V^{-1}(\nu(g_2)) dg_1dg_2. \]
Let $f^{(s)}$ be the holomorphic section of $I_n^n(s,\chi_V)=Ind_{P(Y_n)}^{G_{2n}}\chi_V\alpha_E^s$. Set
\[Z(s,f^{(s)};f_1,f_2)=\int_{[U(W)\times U(W)]} E^{(n,n)}(f^{(s)})(g_1,g_2)\cdot \overline{f_1(g_1)}f_2(g_2)\chi_V^{-1}(\nu(g_2))dg_1dg_2. \]
Thus
\[\langle\theta_{n,r}(\phi_1,f_1),\theta_{n,r}(\phi_2,f_2) \rangle=c\cdot Val_{s=(m-n)d/2}Z(s,\Phi^{n,r}(\delta(\phi_1\otimes\overline{\phi_2}));f_1,f_2) \]
where $\Phi^{n,r}(\delta(\phi_1\otimes\overline{\phi_2}))$ is the Siegel-Weil section associated with $\delta(\phi_1\otimes\overline{\phi_2})$.

For $Re(s)$ sufficiently large, if $f^{(s)}=\otimes_vf_v^{(s)} $ and $f_i=f_{i,v}$ are pure tensors, one has an Euler product
\[Z(s,f^{(s)};f_1,f_2)=\prod_{v}Z_v(s,f_v^{(s)};f_{1,v},f_{2,v}) \]
where $$Z_v(s,f_v^{(s)};f_{1,v},f_{2,v})=\int_{U(W)(F_v)}f_v^{(s)}(g_v,1)\cdot\overline{\langle \pi_v(g_v)f_1,f_2\rangle}dg_v. $$
It gives us the standard $L$-function $L(s+1/2,\pi_v\times\chi_{V,v})$.
If every data involved is unramified (which is the case for almost all $v$), then one has
\[Z_v(s,f_v^{(s)};f_{1,v},f_{2,v})=L(s+1/2,\pi_v\times\chi_{V,v})/b_v(s,\chi_{V}). \]
Note that when $s>0$, $b_v(s,\chi_{V})$ has no poles and the Euler product $b(s,\chi_V)$ is absolutely convergent. In general, we would like to define the normalized local zeta integral
\[Z^\ast_v(s,f_v^{(s)};f_{1,v},f_{2,v})=\frac{Z_v(s,f_v^{(s)};f_{1,v},f_{2,v})}{L(s+1/2,\pi_v\times\chi_{V,v})}. \]
By the hypothesis that $E_v$ splits at all archimedean places $v|\infty$, $Z^\ast_v(s,f_v^{(s)};f_{1,v},f_{2,v})$ at $s=s_0=(m-n)d/2$ is nonzero if and only if the local theta lift
$\Theta_{n,r}(\pi_v)\neq0$. Then Theorem \ref{fininaltheoremRallis} holds due to  the following equality
\[\langle \theta_{n,r}(\phi_1,f_1),\theta_{n,r}(\phi_2,f_2)\rangle=c\cdot Val_{s=s_0}L(s+1/2,\pi\times \chi_V)\cdot Z^\ast(s,\Phi^{n,r}(\delta(\phi_1\otimes\overline{\phi_2}));f_1,f_2)  \] 
where $Z^\ast(s,f^{(s)};f_1,f_2)=\prod_v Z^\ast_v(s,f_v^{(s)};f_{1,v},f_{2,v})$ is absolutely convergent.
\end{proof}
\begin{rem}
	If \cite[Conjecture 11.4]{gan2014siegelweil} holds, then we can remove the assumption that $E_v=F_v\oplus F_v$ for all archimedean places $v|\infty$.
\end{rem}

\subsection*{Acknowledgments} The author would like to thank Yuanqing Cai for helpful discussions.
\bibliography{seigel}
\bibliographystyle{plain}

\end{document}